\documentclass[a4paper]{article}

\usepackage{amsmath}
\usepackage{amssymb}
\usepackage{amsthm}
\usepackage[T1]{fontenc}
\usepackage{mathrsfs}

\newcommand{\Z}{\mathbb{Z}}

\newcommand{\R}{\mathbb{R}}
\newcommand{\C}{\mathbb{C}}

\newcommand{\tensor}{\otimes}
\newcommand{\Tensor}{\bigotimes}
\newcommand{\vect}[1]{\mathbf{#1}}

\newcommand{\scalprod}[1]{\left<  #1 \right> }
\newcommand{\noiseprod}[1]{\left<  #1 \right>_{\!\mu } }
\DeclareMathOperator*{\E}{\mathbb{E}}
\DeclareMathOperator*{\Var}{Var}
\DeclareMathOperator{\Inf}{Inf}

\newcommand{\eps}{\epsilon}

\theoremstyle{plain}
\newtheorem{theorem}{Theorem}[section]
\newtheorem{lemma}[theorem]{Lemma}
\newtheorem{corollary}[theorem]{Corollary}

\newtheorem{question}[theorem]{Question}
\newtheorem{proposition}[theorem]{Proposition}

\theoremstyle{definition}
\newtheorem{definition}[theorem]{Definition}
\newtheorem{fact}[theorem]{Fact}
\newtheorem{remark}[theorem]{Remark}

\newcommand{\ignore}[1]{}

\newcommand{\sectlabel}[1]{\label{section:#1}}
\newcommand{\sectref}[1]{\mbox{Section \ref{section:#1}}}
\newcommand{\theoremlabel}[1]{\label{theorem:#1}}
\newcommand{\theoremref}[1]{\mbox{Theorem \ref{theorem:#1}}}
\newcommand{\lemmalabel}[1]{\label{lemma:#1}}
\newcommand{\lemmaref}[1]{\mbox{Lemma \ref{lemma:#1}}}
\newcommand{\proplabel}[1]{\label{proposition:#1}}
\newcommand{\propref}[1]{\mbox{Proposition \ref{proposition:#1}}}
\newcommand{\corollabel}[1]{\label{corollary:#1}}
\newcommand{\corolref}[1]{\mbox{Corollary \ref{corollary:#1}}}
\newcommand{\factlabel}[1]{\label{fact:#1}}
\newcommand{\factref}[1]{\mbox{Fact \ref{fact:#1}}}

\newcommand{\questlabel}[1]{\label{conjecture:#1}}
\newcommand{\questref}[1]{\mbox{Question \ref{conjecture:#1}}}
\newcommand{\deflabel}[1]{\label{definition:#1}}
\newcommand{\defref}[1]{\mbox{Definition \ref{definition:#1}}}
\newcommand{\eqnlabel}[1]{\label{equation:#1}}
\newcommand{\eqnref}[1]{\mbox{Equation (\ref{equation:#1})}}
\newcommand{\figlabel}[1]{\label{figure:#1}}
\newcommand{\figref}[1]{\mbox{Figure \ref{figure:#1}}}

\begin{document}
\title{Noise Correlation Bounds for Uniform Low Degree Functions}
  \author{ Per Austrin\thanks{ E-mail: \texttt{austrin@kth.se}.
  Research funded by ERC Advanced investigator grant 226203 and a
  grant from the Mittag-Leffler Institute.  Work done in part while
  the author was visiting U.C.\ Berkeley under a grant from the
  Swedish Royal Academy of Sciences.
    }\\
    KTH -- Royal Institute of Technology\\
    Stockholm, Sweden
    \and
    Elchanan Mossel\thanks{
      E-mail: \texttt{mossel@stat.berkeley.edu}.
Research supported by BSF grant 2004105, NSF CAREER award DMS 0548249, DOD ONR grant N0014-07-1-05-06 and ISF grant 1300/08} \\
U.C.\ Berkeley, U.S.A And \\ Weizmann Institute of Science, Rehovot, Israel\\}

\maketitle

\begin{abstract}
We study correlation bounds under pairwise independent distributions
for functions with no large Fourier coefficients.  Functions in which
all Fourier coefficients are bounded by $\delta$ are called
$\delta$-{\em uniform}.  The search for such bounds is motivated by their potential applicability to hardness
of approximation, derandomization, and additive combinatorics.

In our main result we show that $\E[f_1(X_1^1,\ldots,X_1^n) \ldots f_k(X_k^1,\ldots,X_k^n)]$ is close to
$0$ under the following assumptions:
\begin{itemize}
\item
The vectors $\{ (X_1^j,\ldots,X_k^j) : 1 \leq j \leq n\}$ are i.i.d, and for each $j$
the vector $(X_1^j,\ldots,X_k^j)$ has a pairwise independent distribution.

\item
The functions $f_i$ are uniform.
\item
The functions $f_i$ are of low degree.
\end{itemize}

We compare our result with recent results by the second
author for low influence functions and to recent results in additive combinatorics using the
Gowers norm. Our proofs
extend some techniques from the theory of hypercontractivity to a
multilinear setup.

\end{abstract}

\clearpage

\section{Introduction}
\subsection{Functionals of Pairwise Independent Distributions}
\sectlabel{smoothness}

In recent years there has been an extensive study of conditions satisfied by functions $f_1,\ldots,f_k$ which guarantee that
\begin{equation} \label{eq:corr_eq}
\E[f_1(X_1) \cdots f_k(X_k)] \approx \prod_{i=1}^k \E[f_i(X_i)],
\end{equation}
for certain probability distributions over $(X_1,\ldots,X_k)$ that are
pairwise independent.  Recall that the random vector $(X_1,\ldots,X_k)$
is {\em pairwise independent} if for all $1 \leq i < j \leq k$ the
random variables $X_i$ and $X_j$ are independent.  In the current
paper we will consider this problem under the additional assumption
that for all $1 \leq i \leq k$ the random variable $X_i$ is an $n$
dimensional vector $X_i = (X_i^1,\ldots,X_i^n) \in \Omega^n$ and that
$(X_1^j,\ldots,X_k^j)$ follow the same (pairwise independent)
distribution $\mu$ over $\Omega^k$, independently for each $1 \leq j
\leq n$ (see \figref{randommatrix}).
We further assume that $\Omega$ is a finite probability space.

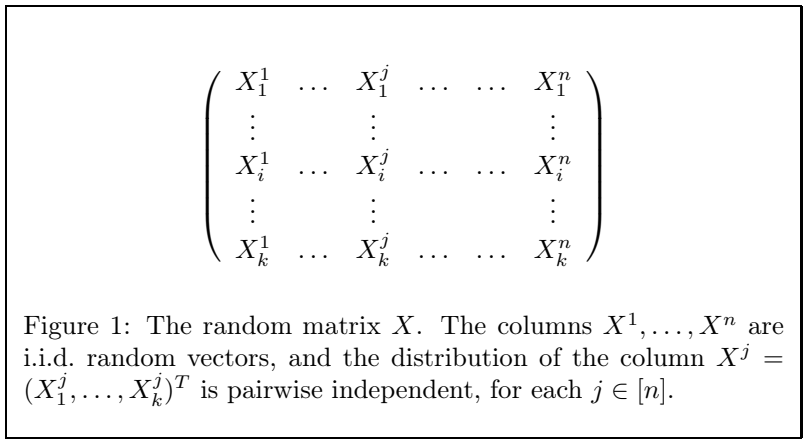
\begin{figure}[!h]
\begin{center}
\fbox{
\begin{minipage}{10cm}
\[
\]
  $$
  \left(
  \begin{array}{cccccc}
    X_1^1   & \ldots & X_1^j & \ldots & \ldots & X_1^n \\
    \vdots  &        & \vdots&        & & \vdots  \\
    X_i^1   & \ldots & X_i^j & \ldots & \ldots & X_i^n \\
    \vdots  &        & \vdots&        & & \vdots \\
    X_k^1   & \ldots & X_k^j & \ldots & \ldots & X_k^n
  \end{array}
  \right)
  $$

  \caption{The random matrix $X$.  The columns $X^1, \ldots, X^n$ are
    i.i.d.\ random vectors, and the distribution of the column $X^j = (X^j_1, \ldots, X^j_k)^{T}$ is pairwise independent, for each $j \in [n]$.}
  \figlabel{randommatrix}\
  \end{minipage}
  }
  \end{center}
\end{figure}

In most of the paper we focus on the related problem of finding
conditions which guarantee that
\begin{equation}
  \label{eq:basic_eq}
  \E\left[\prod_{i=1}^k f_i(X_i)\right] \approx 0,
\end{equation}
which can be thought of as the special case of (\ref{eq:corr_eq}) when
$\prod \E[f_i] \approx 0$.  In many cases a general bound of
type~(\ref{eq:corr_eq}) is straight-forward to obtain from
(\ref{eq:basic_eq}).

The basic example of a condition implying~(\ref{eq:basic_eq}) is one
of the constituents of the proof of Roth's
theorem~\cite{roth53certain}.\footnote{Roth's original argument
  considers $(X_1,X_2,X_3)$ which is a uniformly chosen $3$-term
  arithmetic progression in $\Z_p$ but the argument extends
  immediately to the setup considered here.}  Indeed, it is not too
hard to show that
\begin{equation} \label{eq:roth}
\left|\E\left[\prod_{i=1}^3 f_i(X_i)\right]\right| \leq \min_{1 \leq i \leq 3} \| \hat{f}_i \|_{\infty}.
\end{equation}
where
\begin{itemize}
\item
$(X_1,X_2,X_3)$ are pairwise independent.
\item
$f_1,f_2,f_3$ are any functions with $\max_{1 \leq i \leq 3} \| f_i \|_2 \leq 1$
and $\hat f_1$, $\hat f_2$, $\hat
f_3$ are their Fourier transforms.
\end{itemize}

Gowers~(\cite{Gowers:01}, Theorem 3.2) generalized~(\ref{eq:roth}) and showed that:
\begin{equation} \label{eq:gowers}
\left|\E\left[\prod_{i=1}^k f_i(X_i)\right]\right| \leq \min_{1 \leq i \leq k} \|f_i\|_{U^{k-1}}
\end{equation}
where
\begin{itemize}
\item
$(X_1,\ldots,X_k)$ is a uniformly chosen $k$-term arithmetic progression in $\Z_p^n$.
\item
The functions $f_i$ are all bounded by $1$.
\item
$\|f\|_{U^d}$ is the $d$'th Gowers norm of $f$ (see \defref{gowersnorm}).
\end{itemize}

Note that the uniform distribution over arithmetic progressions
$X_1,\ldots,X_k$ of length $3 \le k \le p$ defines a pairwise
independent distribution in $(\Z_p^n)^k$.  See also~\cite{GreenTao:08}
and~\cite{gowers10true} where more general results are obtained for
other pairwise independent distributions which are defined by linear
equations.

Apart for the additive context, expressions of the form $\prod_{i=1}^k
f_i(X_i)$ often appear in the study of hardness of approximation in
computer science.  In this context, a natural condition is that the
functions $f_1,\ldots,f_k$ all have {\em low influences}.  For
example, recent results of Samorodnitsky and Trevisan
(\cite{samorodnitsky05gowers}, Lemma 8) show how to utilize the Gowers
norms in order to show that (here, $\Inf_j(f_i)$ is the influence of
$X_i^j$ on $f_i$, see e.g.\ \cite{samorodnitsky05gowers} for the exact
definition):

\begin{equation} \label{eq:TS}
\left| \E\left[\prod_{i=1}^{k} f_i(X_i)\right] \right| \leq
O \left( \sqrt{\max_{1 \leq i \leq k} \max_{1 \leq j \leq n} \Inf_j(f_i)} \right)
\end{equation}
provided that:
\begin{itemize}
\item
$k=2^d$ and $X_1,\ldots,X_{k}$ are the elements of a uniformly chosen $d$-dimensional subspace of $\Z_2^n$
\item The functions $f_i$ are all bounded by $1$, and at least one of
  them has $\E[f_i] = 0$
\end{itemize}

As a special case this result gives a so-called ``inverse theorem''
for the $d$'th Gowers norm showing that any function with large $d$'th
Gowers norm must have an influential variable.  The result also
allowed the authors to obtain computational inapproximability results
for certain constraint satisfaction problems, assuming the so-called
Unique Games Conjecture \cite{khot02power}.  The results of
~\cite{samorodnitsky05gowers} include a more general statement which
applies in any product group.

A more recent result of the second author (\cite{Mossel:09}, Theorem
1.14) derives a bound similar
to~(\ref{eq:TS}) by showing:
\begin{equation} \label{eq:MOS}
\left| \E\left[\prod_{i=1}^{k} f_i(X_i)\right] \right| \leq
\Psi^{-1} \left( \max_{1 \leq i \leq k} \max_{1 \leq j \leq n} \Inf_j(f_i) \right),
\end{equation}
where $\Psi(\eps) = \eps^{O(\log(1/\eps)/\eps)}$, provided that:
\begin{itemize}
\item The distribution $\mu$ of $(X_1^j,\ldots,X_k^j)$ is any {\em
    connected} pairwise independent distribution.  This means that for
  every $x,y$ in the support of the distribution there exists a path
  from $x$ to $y$ in the support that is obtained by flipping one
  coordinate at a time.
\item The functions $f_i$ are all bounded by $1$ and at least one of
  them has $\E[f_i] = 0$.
\end{itemize}
The proof of~(\ref{eq:MOS}) is based on showing that if all functions $f_i$ are of degree at most $d$ then (\cite{Mossel:09}, Theorem 4.1):
\begin{equation} \label{eq:MOS2}
\left| \E\left[\prod_{i=1}^{k} f_i(X_i)\right] \right| \leq
C^d \sqrt{\max_{1 \leq i \leq k} \max_{1 \leq j \leq n} \Inf_j(f_i) }
\end{equation}
for some absolute constant $C$ provided that:
\begin{itemize}
\item
The distribution $\mu$
of $(X_1^j,\ldots,X_k^j)$ is {\em any} pairwise independent distribution.
\item
The functions $f_i$ satisfy $\| f_i \|_2 \leq 1$ for all $i$ and at least one of them has $\E[f_i] = 0$.
\end{itemize}
The bound~(\ref{eq:MOS}) is then derived from~(\ref{eq:MOS2}) by
applying certain truncation arguments.  These results
of~\cite{Mossel:09} do not use any algebraic symmetries or the Gowers
norm.  Rather, they were based on extending Lindeberg's proof of
the CLT~\cite{Lindeberg:22} using invariance and generalizing recent
work~\cite{Rotar:79,MoOdOl:09}.

We note that the results of~\cite{Mossel:09} later implied results by the
authors of this paper~\cite{austrin09approximation} which gave
stronger and more general inapproximability results than those
obtained in~\cite{samorodnitsky05gowers}.  It was further noted
in~\cite{Mossel:09} that many of the additive applications involve
pairwise independent distributions.

\subsection{Our Results}
Motivated by these lines of work in additive number theory and
hardness of approximation we wish to obtain weaker conditions that
guarantee~(\ref{eq:basic_eq}). Indeed our main result,
\theoremref{corrbound_main}, shows that
\begin{equation} \label{eq:main}
\left|\E\left[\prod_{i=1}^k f_i(X_i)\right] \right| \leq C^d \| \widehat{f_1} \|_{\infty} \prod_{i=2}^k \| f_i \|_2
\end{equation}
for some constant $C$ which only depends on the pairwise independent distribution $\mu$, where
\begin{itemize}
\item $\|\widehat{f_1}\|_{\infty} = \max |\widehat{f_1}(\sigma)|$
denotes the size of the largest Fourier coefficient of $f_1$.
\item
$(X_1,\ldots,X_k)$ is pairwise independent as in \figref{randommatrix}.
\item
The functions $f_i$ are of Fourier degree at most $d$. In other words, all of their Fourier coefficients at levels above $d$ are $0$.
\end{itemize}

We also give some basic extensions of this.  As a first simple
corollary we give in \corolref{invariance_expectation} a result of
type~(\ref{eq:corr_eq}) with similar error bounds as our main theorem.
Elaborating on this extension we show in
\corolref{invariance_expectation3} that in the case
when~(\ref{eq:corr_eq}) does not hold, one can find three Fourier
coefficients $\widehat{f_{i_1}}(\sigma_1)$,
$\widehat{f_{i_2}}(\sigma_2)$ and $\widehat{f_{i_3}}(\sigma_3)$ which
are all of non-negligible magnitude, and which ``intersect'' in the
sense that $\sigma_1$, $\sigma_2$ and $\sigma_3$ share some variable
$j \in [n]$.  Results of this type are often useful in applications to
hardness of approximation.

We note that the conditions on the underlying distribution and
uniformity are very weak while the condition on the Fourier degree of
the function is very strong. By a simple application of Hölder's
inequality, we will see in \propref{holder} that the results extend to
functions which are ``almost low-degree'' in the sense that the
high-degree parts have small $\ell_k$ norm.

As mentioned above, the proofs of \cite{Mossel:09} work by first
establishing the result~(\ref{eq:MOS2}) for arbitrary low-degree
polynomials and then performing a truncation argument,
giving~(\ref{eq:MOS}) where the degree requirements have been traded
for an additional requirement on the pairwise independent distribution
(and a requirement that the functions are bounded).  Hence, the work
presented in this paper may be viewed as an important step in
establishing similar results for a wider family of functions.  Note
that our result~(\ref{eq:main}) is strictly stronger
than~(\ref{eq:MOS2}) as the bound is stated in terms of the largest
Fourier coefficient instead of the largest influence (and that it
suffices that only one of the functions has small coefficients, as
opposed to (\ref{eq:MOS2}) where all the functions are required to
have small influences).

A very natural question to ask is to what extent the (rather severe)
degree restriction can be relaxed.  Unfortunately, this restriction
can not be removed completely, since for the pairwise independent
relation corresponding to the Gowers norm, it is known using examples
due to Gowers~\cite{Gowers:98} and Furstenberg and Weiss~\cite{FurstenbergWeiss:96} that there are
functions with large $U^3$ norm but no large Fourier coefficients.
However, it is quite possible that the degree restriction can be
removed provided one is willing to require a bit more of the pairwise
independent distribution.  In particular, if one as in (\ref{eq:MOS})
requires that $\mu$ is \emph{connected}, the counterexample given by
the Gowers norm is excluded.  Such a restriction, while generally too
strong in the additive combinatorics settings, is often quite natural
in applications to hardness of approximation and social choice.



\subsection{Applications}

The applications we present mostly concern functions of low Fourier
degree with no large Fourier coefficients. We show that such functions
cannot ``distinguish'' between truly independent distributions and
pairwise independent product distributions.  In particular we show
that such functions defined over $\Z_p^n$ always have small Gowers
norm.  This implies that for functions of low Fourier degree all of
the $U^k$ norms are equivalent for $k \geq 2$. Moreover, such
functions cannot distinguish the uniform distribution over arithmetic
progressions from the uniform distributions over the product space.

\subsection{Proof Idea}

The proof of~(\ref{eq:main}) is based on induction on the degree and
the number of variables. In a way it is similar to inductive proofs
for deriving hyper-contractive estimates for polynomials of random
variables, see, e.g.,~\cite{MoOdOl:09}. Naturally the setup is
different as each polynomial is applied on different random variables.
The pairwise independence property is crucial in the proof as it shows
that certain second order terms vanish.

\subsection{Paper Structure}
In \sectref{harmanalysis} we recall some background in Fourier analysis and noise correlation.
In \sectref{main} we derive the main result and some corollaries.
In \sectref{app} we derive some applications of the main result.
In \sectref{ext} we discuss potential extensions of the main result.

\section{Preliminaries}
\sectlabel{harmanalysis}

\subsection{Notation}

Let $\Omega$ be a finite set and let $\mu$ be a probability
distribution on $\Omega$.  The following notation will be used
throughout the paper.
\begin{itemize}
\item $(\Omega^n, \mu^{\tensor n})$ denotes the product space $\Omega
\times \ldots \times \Omega$, endowed with the product distribution.
\item $\alpha(\mu) := \min
\{\,\mu(x)\,:\,x \in \Omega, \mu(x) > 0\,\}$ denotes the minimum
non-zero probability of any atom in $\Omega$ under the distribution
$\mu$.
\item
$L^2(\Omega, \mu)$ denotes the space of functions from
$\Omega$ to $\C$.  We define the inner product on $L^2(\Omega, \mu)$
by $\scalprod{f,g} := \E_{x \in (\Omega,
\mu)}[f(x)\overline{g(x)}]$, and the $\ell_p$ norm by $\|f\|_p := (\E_{x \in
(\Omega, \mu)}[|f|^p])^{1/p}$.
\end{itemize}

For a probability distribution $\mu$ on $\Omega_1 \times \ldots \times
\Omega_k$ (not necessarily a product distribution) and $i \in [k]$, we
use $\mu_i$ to denote the marginal distribution on $\Omega_i$.  Such
a distribution $\mu$ is said to be pairwise independent if for every
$1 \le i < j \le k$ and every $a \in \Omega_i$, $b \in \Omega_j$ it
holds that
$$\Pr_{x \in (\Omega_1 \times \ldots \times \Omega_k,
\mu)}[x_i = a \wedge x_j = b] = \mu_i(a) \mu_j(b).$$

\subsection{Fourier Decomposition}

In this subsection we recall some background in Fourier analysis that
will be used in the paper.

Let $q$ be a positive integer (not necessarily a prime power), and let
$(\Omega, \mu)$ be a finite probability space with $|\Omega| = q$,
which is non-degenerate in the sense that $\mu(x) > 0$ for every $x
\in \Omega$.  Let $\chi_0, \ldots, \chi_{q-1} : \Omega \rightarrow \C$
be an orthonormal basis for the space $L^2(\Omega, \mu)$ w.r.t.\ the
scalar product $\scalprod{\cdot, \cdot}$.  Furthermore, we require
that this basis has the property that $\chi_{0} = \vect{1}$, i.e., the
function that is identically $1$ on every element of $\Omega$.

We remark that since the choice of basis is essentially arbitrary, one
can take $\chi_0, \ldots, \chi_{q-1}$ to be an $\R$-valued basis
rather than a $\C$-valued one (which can be desirable in the case when
one works exclusively with $\R$-valued functions).  The only place in
the paper where this distinction makes a difference is the final part
of \theoremref{corrbound_main}, where this is stated explicitly.

In the complex valued case when $\mu$ is the uniform distribution we
can take the standard Fourier basis $\chi_y(x) = \exp(2\pi i x y/ q)$
where we identify $\Omega$ with $\Z_q$ in some canonical way.

For $\sigma \in \Z_q^n$, define $\chi_{\sigma}: \Omega^n \rightarrow
\C$ as $\Tensor_{i \in [n]} \chi_{\sigma_i}$, i.e.,
$$
\chi_{\sigma}(x_1, \ldots, x_n) = \prod_{i \in [n]} \chi_{\sigma_i}(x_i).
$$
It is well-known and easy to check that the functions
$\{\chi_\sigma\}_{\sigma \in \Z_q^n}$ form an orthonormal basis for
the product space $L^2(\Omega^n, \mu^{\tensor n})$.  Thus, every
function $f \in L^2(\Omega^n, \mu^{\tensor n})$ can be written as
$$
f(x) = \sum_{\sigma \in \Z_q^n} \hat{f}(\sigma) \chi_\sigma(x),
$$ where $\hat{f}: \Z_q^n \rightarrow \C$ is defined by
$\hat{f}(\sigma) = \scalprod{f, \chi_\sigma}$.  The
most basic properties of $\hat{f}$ are summarized by
\factref{fourier_basic}, which is an immediate consequence of
the orthonormality of $\{\chi_\sigma\}_{\sigma \in \Z_q^n}$.
\begin{fact}
  \factlabel{fourier_basic}
  We have
  \begin{align*}
    \E[fg] &= \sum_{\sigma} \hat{f}(\sigma) \hat{g}(\sigma), &
    \E[f] &= \hat{f}(\vect{0}), &
    \Var[f] &= \sum_{\sigma \ne \vect{0}}\hat{f}(\sigma)^2.
  \end{align*}
\end{fact}

We refer to the transform $f \mapsto \hat{f}$ as the Fourier
transform, and $\hat{f}$ as the Fourier coefficients of $f$.  We
remark that the article ``the'' is somewhat inappropriate, since the
transform and coefficients in general depend on the choice of basis
$\{\chi_i\}_{i \in \Z_q}$.  However, we will always be working with
some fixed (albeit arbitrary) basis, and hence there should be no
ambiguity in referring to the Fourier transform as if it were unique.
Furthermore, most of the important properties of $\hat{f}$ are
actually basis-independent.  In particular
Definitions~\ref{definition:multiindex}-\ref{definition:lowdegpart} and
\factref{chi_inf_bound} do not depend on the choice of Fourier basis.

Before proceeding, let us introduce some useful notation in relation
to the Fourier transform.

\begin{definition}
  \deflabel{multiindex}
  A \emph{multi-index} is a vector $\sigma \in \Z_q^n$, for some $q$
  and $n$.  The \emph{support} of a multi-index $\sigma$ is $S(\sigma) =
  \{\,i\,:\,\sigma_i > 0\,\} \subseteq [n]$. We extend notations defined for
  $S(\sigma)$ to $\sigma$ in the natural way, and write e.g.\
  $|\sigma|$ instead of $|S(\sigma)|$, $i \in \sigma$ instead of $i
  \in S(\sigma)$, and so on.
\end{definition}

\begin{definition}
  \deflabel{fourierdegree}
  The \emph{(Fourier) degree} $\deg(f)$ of $f \in L^2(\Omega^n, \mu^{\tensor
  n})$ is the infimum of all $d \in \Z$ such that $\hat{f}(\sigma) =
  0$ for all $\sigma$ with $|\sigma| > d$.
\end{definition}

The degree of $f$ is one of its most important properties.  In
general, the smaller $\deg(f)$ is, the more ``nicely behaved'' $f$ is.
When $\deg(f) \le d$, we will refer to $f$ as a \emph{degree-$d$
polynomial} in $L^2(\Omega^n, \mu^{\tensor n})$.

\begin{definition}
  \deflabel{lowdegpart}
  For $f: \Omega^n \rightarrow \C$ and $d \in \Z$, the
  function $f^{\le d}: \Omega^n \rightarrow \C$ is defined by
  $$f^{\le d} = \sum_{|\sigma| \le d} \hat{f}(\sigma) \chi_\sigma.$$
  We define $f^{< d}$, $f^{= d}$, $f^{> d}$ and $f^{\ge d}$
  analogously.
\end{definition}

Another fact which is sometimes useful is the following trivial bound
on the $\ell_{\infty}$ norm of $\chi_{\sigma}$ (recall that
$\alpha(\mu)$ is the minimum non-zero probability of any atom in
$\mu$).

\begin{fact}
  \factlabel{chi_inf_bound}
  Let $(\Omega^n, \mu^{\tensor n})$ be a product space with Fourier
  basis $\{\chi_\sigma\}_{\sigma \in \Z_q^n}$.  Then for any $\sigma \in \Z_q^n$,
  $$ \|\chi_{\sigma} \|_{\infty} \le \alpha(\mu)^{-|\sigma|/2}.
  $$
\end{fact}

\subsection{Noise Correlation}
\sectlabel{noisecorrelation}

In this section we introduce the notion of a noisy inner product and
noise correlation.

Various special cases of noise correlation have been the focus of much
work, as we discuss below.  Informally, the noise correlation between
two functions $f$ and $g$ measure how much $f(x)$ and $g(y)$ correlate
on random inputs $x$ and $y$ which are correlated.  We remark that the
name ``noise correlation'' is a slight misnomer and that ``correlation
under noise'' would be a more descriptive name---we are not looking at
how well a random variable correlates with noise, but rather how well
a collection of random variables correlate with each other in the
presence of noise.

\begin{definition}
  \deflabel{noisyinner} Let $(\Omega, \mu)$ be a product space with
  $\Omega = \Omega_1 \times \ldots \times \Omega_k$, and let $f_1,
  \ldots, f_k$ be functions with $f_i \in L^2((\Omega_i)^n,
  (\mu_i)^{\tensor n})$.  The \emph{noisy inner product} of $f_1,
  \ldots, f_k$ with respect to $\mu$ is
  $$
  \noiseprod{f_1, f_2, \ldots, f_k} = \E\left[ \prod_{i=1}^k f_i \right].
  $$
  The \emph{noise correlation} of $f_1, \ldots, f_k$ with respect to
  $\mu$ is
  $$
  \noiseprod{f_1, f_2, \ldots, f_k} - \prod_{i=1}^k \E\left[f_i\right]
  $$
\end{definition}

As it can take some time to get used to \defref{noisyinner}, let us
write out $\noiseprod{f_1, \ldots, f_k}$ more explicitly.  Let $f_i:
\Omega_i^n \rightarrow \C$ be functions on the product space
$\Omega_i^n$, and let $\mu$ be some probability distribution on
$\Omega = \Omega_1 \times \ldots \times \Omega_k$.  Then,
$$
\noiseprod{f_1, \ldots, f_k} = \E_{X}\left[ \prod_{i=1}^k f_i(X_i) \right],
$$
where $X$ is a $k \times n$ random matrix such that each column of $X$
is a sample from $(\Omega, \mu)$, independently of the other columns,
and $X_i$ refers to the $i$th row of $X$.

The notation $\noiseprod{f_1, \ldots, f_k}$ is a new notation for
quantities studied before in e.g.\ \cite{Mossel:09}, its
applications~\cite{austrin09approximation,raghavendra08optimal} and
additive number theory.  The focus of the current paper is where
$X_1,\ldots,X_k$ are pairwise independent though noisy inner products
are of much interest also in cases for non pairwise independent
distributions including in percolation, theoretical computer science
and social choice, see
e.g.~\cite{benjamini99noise,odonnell03computational,khot07optimal,MoOdOl:09}.

\subsubsection{The Gowers Norm}

An instance of the noisy inner product which has been the focus of
much attention in recent years is the Gowers norm, which we will now
define.  Let $p$ be a prime.  For a function $f: \Z_p^n \rightarrow
\C$ and a ``direction'' $Y \in \Z_p^n$, the ``derivative'' of $f$ in
direction $Y$, $f_Y: \Z_p^n \rightarrow \C$ is defined by $f_Y(X) =
f(X+Y)\overline{f(X)}$.  Repeating, we define $f_{Y_1, \ldots, Y_d}(X)
= (f_{Y_1, \ldots, Y_{d-1}})_{Y_d}(X) = \prod_{S \subseteq [d]}
\mathcal{C}^{|S|+1} f\left(X + \sum_{i \not\in S} Y_i\right)$, where
$\mathcal{C}$ denotes the complex conjugation operator.

\begin{definition}
  \deflabel{gowersnorm}
  Let $f: \Z_p^n \rightarrow \C$.  The $d$'th Gowers norm of $f$,
  denoted $\|f\|_{U^d}$, is defined by
  $$
  \|f\|_{U^d}^{2^d} = \E\left[f_{Y_1, \ldots, Y_{d}}(X) \right],
  $$ where the expected value is over a random $X \in \Z_p^n$ and $d$
  random directions $Y_1, \ldots, Y_d$.
\end{definition}

This norm was introduced by Gowers~\cite{Gowers:01} in a
Fourier-analytic proof of Szemerédi's Theorem \cite{szemeredi75sets}
and has since been used extensively in additive number theory.  The
Gowers norm can be written as a noisy inner product.  Indeed, we can
write
$$
\|f\|_{U^d}^{2^d} = \E\left[\prod_{S \subseteq [d]} g_S(X_S)\right] = \noiseprod{g_{\emptyset}, \ldots, g_{[d]}}
$$ where we define $g_S: \Z_p^n \rightarrow \C$ by $g_S(X) =
\mathcal{C}^{|S|+1} f(X)$, and the collection $(X_S)_{S \subseteq
[d]}$ of random variables is defined by $X_S = X + \sum_{i \not\in S}
Y_i$, for a uniformly random $X \in \Z_p^n$ and independent uniformly
random directions $Y_1, \ldots, Y_d \in \Z_p^{n}$.

\subsubsection{Noisy Inner Products Under Pairwise Independence}

This paper focuses on noisy inner products under pairwise independent
distributions.  The interest in this special case comes from
applications in computer science and additive number theory. We
briefly mention a few of these applications.

\begin{itemize}
\item
In computer science there is interest in pairwise independent distributions in hardness of approximation,
in particular those of small support. See~\cite{austrin09approximation} where the results of~\cite{Mossel:08,Mossel:09} were used to derive hardness results based on pairwise independence.
\item
As mentioned above, the Gowers norm and the Gowers inner-product are
both noisy inner products.  Note that the
collections of vectors $(X + \sum_{i \in S} X_i : S \subseteq [d])$ is
pairwise (in fact $3$-wise as long as $d \ge 2$) independent.

\item
Another noisy inner product that is closely related to additive applications is obtained by considering arithmetic progressions.
For concreteness consider again the case where all the functions are of $\Z_p^n \to \{0,1\}$ and let $k < p$.
Given $k$ such functions $f_1,\ldots,f_k$ we let:
\[
\noiseprod{f_1,\ldots,f_k} = \E\left[\prod_{i=1}^k f_i(i X + Y)\right],
\]
where $X,Y$ are independent and uniformly chosen in $\Z_p^n$ (note
that $iX+Y$ and $jX+Y$ are independent for $1 \le i < j \le k$).  If
$A$ is an indicator of a set then the number of $k$-term progressions
in $A$ is in fact:
\[
p^{2n} \noiseprod{A, A, \ldots, A}.
\]
\end{itemize}

\section{Main Theorem}
\sectlabel{main}

\newcommand{\degtwo}{\deg_{-2}}

In this section, we state and prove our main theorem.  First we
define the parameter which controls how good bounds we get.

\begin{definition}
  Let $f_1, \ldots, f_k$ be a collection of functions.  We denote by
  $\degtwo(f_1, \ldots, f_k)$ the sum of the $k-2$ smallest degrees of
  $f_1, \ldots, f_k$.
\end{definition}

We can now state the main theorem.

\begin{theorem}
  \theoremlabel{corrbound_main} Let $(\Omega, \mu)$ be a pairwise
  independent product space $\Omega = \Omega_1 \times \ldots \times
  \Omega_k$.  There is a constant $C$ depending only on $\mu$ such
  that the following holds.

  Let $f_1, \ldots, f_k$ be functions $f_i \in L^2(\Omega_i^n,
  (\mu_i)^{\tensor n})$.  Denote by $\delta := \max_{\sigma \in
  \Z_q^n} |\hat{f}_1(\sigma)|$ the size of the largest Fourier
  coefficient of $f_1$, and let $D := \degtwo(f_1, \ldots, f_k)$ denote
  the sum of the $k-2$ smallest degrees of $f_1, \ldots, f_k$.  Then,
  $$
  |\noiseprod{f_1, \ldots, f_k}|
  \le C^{D} \delta  \prod_{i=2}^k \|f_i\|_2.
  $$ Furthermore, one can always take $C = \left(k \sqrt{\frac{q-1}{\alpha}} \right)^3$, where $\alpha = \min_i \alpha(\mu_i)$.  If $\mu$ is balanced,
  i.e., if all marginals $\mu_i$ are uniform, then there is a choice
  of complex Fourier basis such that one can take $C = (k \sqrt{q-1})^3$.
\end{theorem}

We remark that, while \theoremref{corrbound_main} is very limited
because of its requirement on the degrees of the $f_i$'s, the lack of
any other assumptions is nice.  In particular, we do not need to
assume that the $f_i$'s are bounded, nor do we need any assumptions on
$\mu$ beyond the pairwise independence condition.

\begin{proof}
  We prove this by induction over $n$.  If $n = 0$, the statement is
  easily verified (either $D = -\infty$, or $D = 0$, depending on
  whether one of the functions is $0$ or not).\footnote{We point out
  that $f_i \in L^2(\Omega_i^0, (\mu_i)^{\tensor 0})$ does not
  formally make sense.  However in this case, the appropriate way to
  view $f_i$ is as an element of $L^2(\Omega_i^N, (\mu_i)^{\tensor
  N})$ which only depends on the $n$ first coordinates, for some large
  value of $N$.  In particular, for the case $n=0$ we have that $f_i$
  is a constant.}

  Write $f_i = g_i + h_i$, where
  \begin{align*}
    g_i &= \sum_{1 \not\in \sigma} \hat{f}(\sigma) \chi_\sigma  &
    h_i &= \sum_{1 \in \sigma} \hat{f}(\sigma) \chi_\sigma,
  \end{align*}
  i.e., $h_i$ is the part of $f_i$ which depends on $X^1$ (the first
  column of $X$), and $g_i$ is the part which does not depend on
  $X^1$.  Then
  $$
  \noiseprod{f_1, \ldots, f_k} = \E_X \left[\prod f_i(X_i)\right] = \sum_{T \subseteq [k]}
  \E_{X}\left[\prod_{i \not\in T} g_i(X_i) \prod_{i \in T}
    h_i(X_i) \right].
  $$
  For $T \subseteq [k]$, define
  $$
  E(T) = \E_{X}\left[\prod_{i \not\in T} g_i(X_i) \prod_{i \in T} h_i(X_i) \right].
  $$
  The key ingredient will be the following Lemma, bounding $|E(T)|$.

  \begin{lemma}
    \lemmalabel{fourier_inductive_step}
    Let $\emptyset \subseteq T \subseteq [k]$.  Then:
    \begin{itemize}
    \item If $T = \emptyset$, we have
      $$
      |E(T)| \le C^{D} \delta \prod_{i=2}^k \|g_i\|_2 .
      $$
    \item If $1 \le |T| \le 2$, we have
      $$
      E(T) = 0 .
      $$
    \item If $|T| \ge 3$, we have
      $$
      |E(T)| \le C^{D+2} \left(\frac{\sqrt{(q-1)/\alpha}}{C}\right)^{|T|} \delta \prod_{\substack{i \not\in T\\i \ne 1}} \|g_i\|_2  \prod_{\substack{i \in T\\i \ne 1}} \|h_i\|_2 .
      $$
    \end{itemize}
  \end{lemma}

  Before proving the Lemma, let us see how to use it to finish the proof
  of \theoremref{corrbound_main}.

  Write $\|h_i\|_2 = \tau_i \|f_i\|_2$ for some $\tau_i \in [0,1]$, so
  that $\|g_i\|_2 = \sqrt{1-\tau_i^2} \cdot \|f_i\|_2$ (by orthogonality of
  the Fourier decomposition).  By plugging in the different cases of
  \lemmaref{fourier_inductive_step}, we can then bound
  $\noiseprod{f_1, \ldots, f_k}$ by
  \begin{eqnarray} \nonumber
    \lefteqn{|\noiseprod{f_1, \ldots, f_k}| \le \sum_T |E(T)|} \hphantom{12345}\\ \nonumber
    &\le& C^D \delta \prod_{i=2}^k \|g_i\|_2 + \sum_{|T|\ge 3} C^{D+2}\left(\frac{\sqrt{(q-1)/\alpha}}{C}\right)^{|T|} \delta \prod_{\substack{i \not \in T\\i \ne 1}} \|g_i\|_2 \prod_{\substack{i \in T\\i \ne 1}} \|h_i\|_2 \\ \nonumber
    &=& C^D \delta \prod_{i=2}^k \|f_i\|_2 \times {} \\ \label{par}
  && \hphantom{} {} \Bigg(\prod_{i=2}^k
  \sqrt{1-\tau_i^2} + \sum_{|T| \ge 3} C^2 \left(\frac{\sqrt{(q-1)/\alpha}}{C}\right)^{|T|}
  \prod_{\substack{i \not\in T\\i \ne 1}} \sqrt{1-\tau_i^2}
  \prod_{\substack{i \in T\\i \ne 1}} \tau_i \Bigg).
  \end{eqnarray}
  Hence, it suffices to bound the ``factor'' inside the large
  parenthesis in~(\ref{par}) by $1$ in order to complete the proof of
  \theoremref{corrbound_main}.

  Let $\tau = \max_{i \ge 2} \tau_i$.  Then the factor in~(\ref{par}) can be bounded by
  \begin{equation} \label{eq:err1}
  \sqrt{1-\tau^2} + \tau^2 \sum_{i=3}^{k} {k \choose i} \left(\frac{\sqrt{(q-1)/\alpha}}{C^{1/3}}\right)^i
  \end{equation}
  where in the sum the value of $i$ corresponds to the size of the set $T$ and we assumed that $C > 1$ and then used that, for
  $i \ge 3$, $C^{2-i} \le C^{-i/3}$.  To bound~(\ref{eq:err1}), we use
  the following simple lemma:
  \begin{lemma}
    For every $k \ge 3$,
    $$
    \sum_{i=3}^k {k \choose i} \frac{1}{k^i} \le 1/2.
    $$
  \end{lemma}
  \begin{proof}
    Since ${k \choose i} \le k^i/i!$ we have
    $$ \sum_{i=3}^k {k \choose i} \frac{1}{k^i} \le \sum_{i=3}^k
    \frac{1}{i!} \le e - 5/2 \le 1/2,
    $$ where the second inequality is by the Taylor expansion $e =
    \sum_{i=0}^{\infty} \frac{1}{i!} \ge \sum_{i=0}^k \frac{1}{i!}$.
  \end{proof}

  Hence, if $C \ge \left(k\sqrt{\frac{q-1}{\alpha}}\right)^3$, the
  factor in~(\ref{par}) is bounded by
  $$
  \sqrt{1-\tau^2} + \tau^2 /2 \le 1.
  $$ This concludes the proof of \theoremref{corrbound_main}.  We
  have not yet addressed the claim that if the marginals $\mu_i$ are
  uniform, there is a Fourier basis such that $C$ can be chosen as
  $(k\sqrt{q-1})^3$.  See the comment after the proof of
  \lemmaref{fourier_inductive_step}.
  \end{proof}

  We now prove the lemma used in the previous proof.

  \begin{proof}[Proof of \lemmaref{fourier_inductive_step}]
    The case $T = \emptyset$ is a direct application of the induction
    hypothesis, since the functions $g_i$ depend on at most $n-1$
    variables (and have $\deg_{-2}(g_1, \ldots, g_k) \le D$).

    For $i \in [k]$, write
    $$
    h_i(x) = \sum_{j=1}^{q-1} \chi_{i,j}(x_1) h_{i,j}(x_2, \ldots, x_n)
    $$ for a Fourier basis $\chi_{i,0}=1, \chi_{i,1}, \ldots, \chi_{i,q-1}$ of
    $L^2(\Omega_i, \mu_i)$. Denoting by $X^j$ the $j$th column of
    $X$, and writing $\E_{X^2,\ldots,X^n}$ for the average over $X^2,\ldots,X^n$ we can write $E(T)$ as
    \begin{eqnarray*}
      E(T) &=& \E_{X^2,\ldots,X^n}\left[ \prod_{i \not \in T} g_i(X_i) \E_{X^1}\left[\prod_{i \in T} h_i(X_i)\right] \right] \\
      &=& \E_{X^2,\ldots,X^n}\left[ H_T(X) \cdot \prod_{i \not \in T} g_i(X_i) \right],
    \end{eqnarray*}
    where
    \begin{eqnarray*}
      H_T(X) &=& \E_{X^1}\left[\prod_{i \in T} h_i(X_i)\right] \\
      &=& \sum_{\sigma \in [q-1]^{T}} \E_{X^1} \left[ \prod_{i \in T} \chi_{i,\sigma_i}(X^1_i) \right] \prod_{i \in T} h_{i,\sigma_i}(X_i).
    \end{eqnarray*}
    Now for $1 \le |T| \le 2$, the pairwise independence of $\mu$ gives
    that for any $\sigma \in [q-1]^T$,
    $$
    \E_{X^1} \left[ \prod_{i \in T} \chi_{i,\sigma_i}(X^1_i) \right] = \prod_{i \in T} \E[ \chi_{i,\sigma_i} ] = 0,
    $$
    hence in this case $H_T(X) = 0$ and by extension $E(T) = 0$.

    Thus, only the case $|T| \ge 3$ remains.  By Hölder's
    inequality, we can bound
    \begin{eqnarray}
      \eqnlabel{chiprod_bound}
      \E_{X^1} \left[ \prod_{i \in T} \chi_{i,\sigma_i}(X^1_i) \right] \le
    \prod_{i \in T} \|\chi_{i,\sigma_i}\|_{|T|}.
    \end{eqnarray}
    By \factref{chi_inf_bound}, $\|\chi_{i,\sigma_i}\|_{\infty}$ can be
    bounded by
    \[
    \sqrt{1/\alpha(\mu_i)} \le \sqrt{1/\min_i \alpha(\mu_i)} = \sqrt{1/\alpha}.
    \]
    Hence we can bound the above by
    $(1/\alpha)^{|T|/2}$.

    Plugging this into $E(T)$ gives
    \begin{eqnarray*}
      E(T) &\le& (1/\alpha)^{|T|/2} \E_{X^2,\ldots,X^n}\left[ \sum_{\sigma \in [q-1]^{T}} \prod_{i \in T} h_{i,\sigma_i}(X_i) \prod_{i \not \in T} g_i(X_i) \right].
    \end{eqnarray*}

    For $\sigma \in [q-1]^T$, let $D_\sigma$ be the sum of the $k-2$ smallest
    degrees of the polynomials $\{g_i: i \not\in T\} \cup \{h_{i,\sigma_i}:
    i \in T\}$.  Since $g_i$ and $h_{i,\sigma_i}$ are functions of $n-1$ variables,
    we can use the induction hypothesis to get a bound of
    \begin{eqnarray*}
      E(T) &\le& (1/\alpha)^{|T|/2} \sum_{\sigma \in [q-1]^{T}}  C^{D_\sigma} \delta \prod_{\substack{i \in T\\i \ne 1}} \|h_{i,\sigma_i}\|_2 \prod_{\substack{i \not \in T\\i \ne 1}} \|g_i\|_2.
    \end{eqnarray*}
    But since the $h_{i,\sigma_i}$'s have strictly smaller degrees than the
    corresponding $f_i$'s, $D_\sigma$ is bounded by $D-|T|+2$, and hence we
    have that
    \begin{eqnarray*}
      E(T) &\le& \alpha^{-|T|/2} C^{D-|T|+2} \sum_{\sigma \in [q-1]^{T}}  \delta \prod_{\substack{i \in T\\i \ne 1}} \|h_{i,\sigma_i}\|_2 \prod_{\substack{i \not \in T\\i \ne 1}} \|g_i\|_2\\
      &\le& C^{D+2} \left(\frac{\sqrt{(q-1)/\alpha}}{C}\right)^{|T|} \delta  \prod_{\substack{i \in T\\i \ne 1}} \|h_i\|_2 \prod_{\substack{i \not \in T\\i \ne 1}} \|g_i\|_2,
    \end{eqnarray*}
    where we used the fact that $\sum_{j \in [q-1]} \|h_{i,j}\|_2 \le
    \sqrt{q-1} \|h_i\|_2$ (by Cauchy-Schwarz and orthogonality of the
    functions $h_{i,j}$).

    To obtain the bound for $|E(T)|$, we can simply negate one of
    functions $g_i$ for $i \not\in T$ or $h_i$ for $i \in T$, so that
    $E(T)$ is negated and the calculations above produce an upper
    bound on $-E(T)$.  This concludes the proof of
    \lemmaref{fourier_inductive_step}.
  \end{proof}

  \begin{remark}
    In the case when the marginal distributions $\mu_i$ are uniform, one
    can take as basis of $(\Omega, \mu)$ the standard Fourier basis $\chi_{y}(x) =
    e^{2\pi i \frac{y \cdot x}{q}}$ (where we identify the elements $x$
    of $\Omega$ with $\Z_q$).  For this basis, $\|\chi_{j}\|_{\infty} =
    1$ and hence \eqnref{chiprod_bound} can be bounded by $1$ rather
    than $1/\sqrt{\alpha}$, which implies that for this basis, we can
    choose $C = (k\sqrt{q-1})^3$.
  \end{remark}

\subsection{Corollaries}

We proceed with some corollaries of \theoremref{corrbound_main}.  The
first says that if all non-empty Fourier coefficients of $f_1$ are
small, then the noise correlation is small.

\begin{corollary}
  \corollabel{invariance_expectation}
  Assume the setting of \theoremref{corrbound_main}, but with $\|f_i\|_2 \le 1$ for each $i$ and
  \[
  \delta := \max_{1 \le i \le k-2} \max_{\sigma \ne \vect{0}} |\hat{f}_i(\sigma)|.
  \]
Then,
  \begin{equation} \label{eq:main2}
   \left| \noiseprod{f_1, \ldots, f_k} - \prod_{i=1}^k \E[f_i]\right|
  \le \delta (k-2) C^D,
  \end{equation}
  where $C$ and $D$ are as in \theoremref{corrbound_main}.
\end{corollary}

\begin{proof}
  We prove the claim by induction on $k$. The case $k=2$ is trivial.
  For the induction hypothesis let $g_1(x) = f_1(x) - \E[f_1]$.  Then by \theoremref{corrbound_main}
  $$
  \left|\noiseprod{f_1, \ldots, f_k} - \E[f_1] \noiseprod{f_2, \ldots, f_k} \right| = | \noiseprod{g_1,f_2,\ldots,f_k} |
  \leq \delta C^D
  $$
  and by the induction hypothesis
  $$
  \left|\E[f_1] \noiseprod{f_2, \ldots, f_k} - \prod_{i=1}^k \E[f_i] \right| =
  |\E[f_1]| \cdot \left|\noiseprod{f_2, \ldots, f_k} - \prod_{i=2}^k \E[f_i] \right| \leq (k-3) \delta C^D.
  $$
  The proof follows.
\end{proof}

A more careful examination of the proof above reveals that in the case
where the noise correlation is large there should be a basis element
with large weight in one of the functions that is correlated with some
other functions. Specifically:

\begin{corollary}
  \corollabel{invariance_expectation2} Assume the setting of
  \theoremref{corrbound_main} but with $D = \sum \deg(f_i)$ the sum of
  the degrees of all the functions, and $\| f_i \|_2 \le 1$ for each $f_i$.

Then for all $\delta > 0$ if:
  \begin{equation} \eqnlabel{invariance_expectation2}
   \left| \noiseprod{f_1, \ldots, f_k} - \prod_{i=1}^k \E[f_i]\right|
  > 2 \delta (k-2) C^D,
  \end{equation}
  then there exists an $1 \leq i \leq k-2$ and a non-empty multi-index $\sigma$ such that
  \[
  |\hat{f}_i(\sigma)| > \delta, \quad | \E[\chi^i_{\sigma}  \cdot f_{i+1} \cdots f_k ] | > \delta^2 C^D
  \]
  where $C$ is the constant from \theoremref{corrbound_main}.
\end{corollary}

\begin{proof}
  From the proof of \corolref{invariance_expectation} it follows that if
  \eqnref{invariance_expectation2} holds then there exists an $1 \le i
  \le k-2$ such that
  \[
  |\noiseprod{g_i, f_{i+1}, \ldots, f_{k}}| > 2 \delta C^D,
  \]
  where $g_i = f_i - \E[f_i]$.
  Write $g_i = \sum_{\sigma \in A} \hat{g}_i(\sigma) \chi^i_{\sigma} + h_i$ where $A$ is the set of all $\sigma$
  for which $|\hat{g}_i(\sigma)| > \delta$. Then by \theoremref{corrbound_main} it follows that:
  \[
  |\E[h_i f_{i+1} \cdots f_k]| < \delta C^D,
  \]
  which implies
  \[
  \left|\E\left[\left(\sum_{\sigma \in A} \hat{g}_i(\sigma) \chi^i_{\sigma}\right) f_{i+1} \cdots f_k\right]\right| > \delta C^D.
  \]
  Writing
  \[
  t(\sigma) = \E\left[\chi^i_{\sigma} f_{i+1} \cdots f_k\right],
  \]
  for $\sigma \in A$, we see that $\sum_{\sigma \in A} |\hat{g}_i(\sigma) t(\sigma)| > \delta C^D$.
  Since $\sum_{\sigma \in A}| \hat{g}_i(\sigma)|^2 \leq 1$ it follows that
  \[
  \sum_{\sigma \in A} |\hat{g}_i(\sigma) t(\sigma)| > \delta C^D \sum_{\sigma \in A}| \hat{g}_i(\sigma)|^2,
  \]
  which implies that there exists a $\sigma$ with
  \begin{equation} \eqnlabel{del3}
  |\E\left[\chi^i_{\sigma} f_{i+1} \cdots f_k\right]| = |t(\sigma)| > \delta C^D |\hat{g}_i(\sigma)| \geq \delta^2 C^D.
  \end{equation}
  The proof follows.
\end{proof}

Next we apply the previous corollary to \eqnref{del3} and the
functions $f_{i+1},\ldots,f_k,\chi^i_{\sigma}$ to obtain that
$|\E[f_{j+1} \cdots f_k \chi^i_{\sigma} \chi^j_{\sigma'}]|$ is large
for some $j > i$ and $\sigma'$. Continuing in this manner we obtain
the following:

\begin{corollary}
  \corollabel{invariance_expectation3}
  Assume the setting of
  \theoremref{corrbound_main} but with $D = \sum \deg(f_i)$ the sum of
  the degrees of all the functions, and $\| f_i \|_2 \le 1$ for each $f_i$.

Then for all $\delta > 0$ if:
  \begin{equation} \label{eq:main3}
   \left| \noiseprod{f_1, \ldots, f_k} - \prod_{i=1}^k \E[f_i]\right| > C^D \delta,
  \end{equation}
  then there exists a set $I \subseteq [k]$ with $|I| \geq 3$ and for all $i \in I$ a non-zero multi-index
  $\sigma(i)$ such that:
  \begin{itemize}
  \item
  For all $i \in I$:
  \[
  |\hat{f}_i(\sigma)| > \left(\frac{\delta}{2k}\right)^{2^k}
  \]

  \item
  For all $a \in \cup_{i \in I} S(\sigma(i))$ it holds that
  \[
  |\{ i : a \in S(\sigma(i)) \}| \geq 3
  \]
  (the $3$ above may be replaced by $r+1$ if the distributions involved are $r$-wise independent).
  \end{itemize}
\end{corollary}

\begin{proof}

  Define $\delta_0 = \delta^{1/2}$, and $\delta_{i} =
  \frac{\delta_{i-1}^2}{2k}$.  We show by induction on $r$ that it is
  possible to find $I,J \subseteq [k]$ disjoint where $I$ is of size at
  least $r$ and for all $i \in I$ there exists a non-zero multi-index
  $\sigma(i)$ such that for all $i \in I$:
  \begin{equation} \label{eq:dela}
  |\hat{f}_i(\sigma(i))| > \delta_{r} =
  \frac{\delta^{2^{r-1}}}{(2k)^{2^r-1}} >
  \left(\frac{\delta}{2k}\right)^{2^r}
  \end{equation}
  and further
  \begin{equation} \label{eq:inda}
  \E\left[\prod_{i \in I} \chi^i_{\sigma(i)} \prod_{j \in J} f_j\right]  > C^D \delta_{r+1}.
  \end{equation}
  The base case $r=1$ is established by the previous claim. The
  induction step is proved by noting that if $J$ is non-empty and $j
  \in J$, then we may apply the previous claim to the sequence of
  functions $\{f_j\}_{j \in J}$ followed by the functions
  $\chi^{i}(\sigma(i))$. We then obtain~(\ref{eq:dela})
  and~(\ref{eq:inda}) with $\delta_{r+1}$ and sets $I'$ and $J'$ where
  $J'$ is of size one smaller than $J$.  When we stop with $J =
  \emptyset$ and $r \leq k$ we obtain that $J$ is empty and therefore:
  \[
  \E\left[\prod_{i \in I} \chi^i_{\sigma(i)}\right]  > C^D \delta_{k+1} > 0.
  \]
  This together with pairwise independence implies that
  for all $a \in \cup_{i \in I} S(\sigma(i))$ it holds that
  \[
  |\{ i : a \in S(\sigma(i)) \}| \geq 3
  \]
  as needed.
\end{proof}

We finally note while all of the results above are stated for low-degree polynomials, they also apply for polynomials that are almost low-degree. Indeed H\"{o}lder's inequality implies the following.

\begin{proposition}
  \proplabel{holder}
  Assume the setting of \theoremref{corrbound_main} and with $k$ functions satisfying $\| f_i \|_k \leq 1$ and
  $\| f_i^{>d} \|_k \leq \eps$
  for all $i$. Then
  \[
  \left|\noiseprod{f_1, \ldots, f_k} - \noiseprod{f_1^{\leq d},\ldots,f_k^{\leq d}}\right| \leq k \eps (1+\eps)^{k-1}.
  \]
\end{proposition}

\begin{proof}
The proof follows by using H\"{o}lder's inequality $k$ times, each
time replacing $f_i$ with $f_i^{\leq d}$.  Note that $\| f_i^{\le d}
\|_k \leq \| f_i \|_k + \| f_i^{> d} \|_k \leq 1+\eps$, so that when
making the $i$'th replacement, the error incurred is bounded by
$$\left(\prod_{j=1}^{i-1} \|f_j^{\le d}\|_k \right) \|f_i^{>d}\|_k \left(\prod_{j=i+1}^k
\|f_j\|_k\right) \le (1+\epsilon)^{i-1} \epsilon.$$
\end{proof}

\section{Applications}
\sectlabel{app}

The first application is a ``weak inverse theorem'' for the Gowers norm. From \theoremref{corrbound_main} and the fact that
\[
\|f\|_{U^2} = \left(\sum_{\sigma} |\hat{f}^4(\sigma)|\right)^{1/4}
\]
we immediately obtain that

\begin{proposition} \label{prop:gowers}
Let $f: \Z_p^n \rightarrow \C$ have Fourier degree $d$, have $\| f \|_2 = 1$ and let $k \geq 2$.
If the $k$'th Gowers norm of $f$ satisfies $\|f\|_{U^k} >
\epsilon$, then there exists a multi-index $\sigma \in \Z_p^n$ such
that
$$
|\hat{f}(\sigma)| \geq \left( \frac{\eps}{(2^k\sqrt{q-1})^{3d}} \right)^{2^k},
$$
where the Fourier coefficient is w.r.t.\ the standard Fourier basis.
In particular,
\[
\|f\|_{U^2} \geq \left( \frac{\eps}{(2^k\sqrt{q-1})^{3d}} \right)^{2^k}.
\]
\end{proposition}

This implies that for functions of low Fourier degree, all $U^k$ norms
for constant $k \geq 2$ are equivalent.  We next obtain a similar
result for arithmetic progressions using \theoremref{corrbound_main}
and \corolref{invariance_expectation3}:

\begin{proposition} \label{prop:arith}
Let $(X_1,\ldots,X_k)$ have the uniform distribution over arithmetic progressions of length
 $k$ in $\Z_p^n$, where $3 \le k \le p$. Let $Y_1,\ldots,Y_k$ be i.i.d. and uniformly distributed in
 $\Z_p^n$. Let $f_1,\ldots,f_k: \Z_p^n \rightarrow \C$ have Fourier degree $d$ and $\| f_i \|_2 \leq 1$
 for all $i$.
Then, if
\[
|\E[f_1(X_1) \cdots f_k(X_k)] - \E[f_1(Y_1) \cdots f_k(Y_k)]| > \eps,
\]
it holds w.r.t.\ the standard Fourier basis that:
\begin{enumerate}
\item
None of the functions $f_i$ are $\delta$-uniform with
\[
\delta = \frac{\eps}{(k\sqrt{q-1})^{3dk}}.
\]
\item
There exist indices $1 \leq i(1) < i(2) < i(3) \leq k$ and multi-indices \
\[
\sigma(1), \sigma(2), \sigma(3) \in \Z_p^n, \quad \sigma(1) \cap \sigma(2) \cap \sigma(3) \neq \emptyset,
\]
such that
$$
|\widehat{f_{i(j)}}(\sigma(j))| \ge \left( \frac{\eps}{k \cdot (k \sqrt{q-1})^{3dk}} \right)^{2^k}
$$
for $1 \leq j \leq 3$.
\end{enumerate}
\end{proposition}

We note that the two results above may be interpreted as certain types
of derandomization results which can be defined in further
generality. The basic setup is that there are $2k$ vectors
$X_1,\ldots,X_k$ and $Y_1,\ldots,Y_k$. All of the vectors have the
same distribution which is uniform in some product space $\Omega^n$.
However, the $Y_i$'s are independent while the $X_i$'s are only pairwise
independent. How can the two distributions be distinguished? One way to
distinguish is to consider functions $f_i$ of $X_i$ (resp.~$Y_i$) and to
show that $\prod_{i=1}^k f_i(X_i)$
is far in expectation from $\prod_{i=1}^k f_i(Y_i)$. Our
results show that if the functions $f_i$ are uniform and of low degree
then it is impossible to have such a distinguisher.

We finally note that for all the applications considered here, the
results hold assuming the function is close in the $k$'th norm to
function of low degree by \propref{holder}.

\section{Possible Extensions}
\sectlabel{ext}

We briefly discuss some comments regarding possible extensions of the
main result.

\subsection{Invariance}

The result of~\cite{Mossel:08} shows under stronger conditions the {\em
invariance} of the functions $f_1,\ldots,f_k$.  In other words: they
show that the distribution of $(f_1,\ldots,f_k)$ under the pairwise
distribution is close to the distribution under the product
distribution with the same marginals as $\mu$.

\ignore{ One would not expect that such a strong conclusion will hold
here.  Indeed if $f : \Omega^n \to \R$ is a function with low
influences and $\psi : \R \to \R$ is a ``nice'' functions then the
functions $\psi \circ f$ also has low influences.  However, if $f$ has
low Fourier coefficients - this does not have to be the case for $\psi
\circ f$.  For example consider the function $f : \{-1,1\}^n \to \R$
defined by $(x_1-1)(x_2+...+x_n)/n^{1/2}$.  Then all of the
coefficient of the functions are small with respect to the standard
basis.  However if we look at the function $f^2$ say, then it has a
strong correlation with $x_1$.  Moreover it is possible to find
functions $\psi$ such that $\psi \circ f$ is ``close'' to $x_1$.  So
while it is true that $\E[\prod_{i=1}^k f(X_i)]$ is close to $0$ in
the setup of the current paper - it is easy to construct example where
$\E[\prod_{i=1}^k (\psi \circ f)(X_i)]$ is not close to $\prod_{i=1}^k
\E[\psi \circ f]$.  }

One would not expect that such a strong conclusion will hold here.
Consider for instance the following example.  Let $f : \{-1,1\}^n \to
\R$ be defined by $f(x) = (x_1-1)(x_2+...+x_n)/n^{1/2}$.  Then $f$ has
Fourier degree $2$, variance $\Theta(1)$, and coefficients of order $n^{-1/2}$.  Define a
distribution $\mu$ on triples of strings $(x,y,z) \in (\{-1,1\}^n)^3$,
by letting, for each $i \in [n]$, the distribution on the $i$'th
coordinate be the uniform distribution over $(x_i,y_i,z_i)$ satisfying
$x_i \cdot y_i \cdot z_i = 1$.  Then $\mu$ is balanced pairwise
independent.  Now consider the distribution of $(f(x), f(y), f(z))$,
compared to the distribution of $(f(\tilde x), f(\tilde y), f(\tilde
z))$ for $\tilde x$, $\tilde y$ and $\tilde z$ independent uniformly
random strings of $\{-1,1\}^n$.  The distribution of $(f(x), f(y),
f(z))$ is supported only on points where at least one of the
coordinates is $0$ (since one of $x_1$, $y_1$, $z_1$ is always $1$).
On the other hand, the distribution of $(f(\tilde x), f(\tilde y),
f(\tilde z))$ has an $\Omega(1)$ fraction of its support on points
such that all three of $|f(\tilde x)|$, $|f(\tilde y)|$, and
$|f(\tilde z)|$ are lower bounded by $\Omega(1)$.  Hence the two
distributions are not close, even though the Fourier coefficients of
$f$ can be made arbitrarily small by increasing $n$.

The same reasoning shows that we cannot hope for invariance even if
all moments on up to $k-1$ variables match.  E.g., even if $X_1,
\ldots, X_k$ are $(k-1)$-wise independent it is not necessarily the
case that the distribution of $(f(X_1), \ldots, f(X_k))$ is close to a
product distribution.

\subsection{Relaxed Degree Conditions}

\sectlabel{relaxeddegree}

As mentioned before, previous work~\cite{MoOdOl:09,Mossel:09}
established results of the type discussed here by first deriving the
results for low degree polynomials and then applying ``truncation
arguments'' to obtain results for general bounded functions. It seems
that in the context of the current paper these truncation arguments
are more challenging.

Indeed, it is well-known that in general, large Gowers norm does not imply
large Fourier coefficients (consider e.g.\ the function $f(X) =
(-1)^{\sum_{i=1}^{n-1} x_ix_{i+1}}$ over $\Z_2^n$), and hence one can
not hope to drop the requirement of small Fourier degree and
generalize our theorem to general bounded functions.

However, improvements are still possible.  First, it is possible that
under additional conditions on the pairwise independent marginal
distributions, the requirement on low Fourier degree can be dropped
completely.  We discuss this below.

A second, closely related possible improvement, is to slightly relax
the strong Fourier degree requirements.  In particular, one can hope
that a similar bound can be derived for functions with exponentially
small Fourier tails, i.e., functions $f$ such that the total Fourier
mass on the high-degree part decays exponentially, $\|f^{>d}\|_2^2 \le
(1-\gamma)^d$ for some $\gamma > 0$.  Such functions arise naturally
in many applications, e.g., when functions are evaluated on slightly
noisy inputs.  Hence, it is natural to ask whether the following
extension of our result can be true:

\begin{question}
  \questlabel{extension}
  Let $(\Omega, \mu)$ be a pairwise
  independent product space $\Omega = \Omega_1 \times \ldots \times
  \Omega_k$.  Is it true
  that for every $\gamma > 0$ and $\epsilon > 0$, there exists a
  constant $\delta := \delta(\gamma, \epsilon) > 0$ such that the
  following holds?  If $f_1, \ldots, f_k$ are functions $f_i \in
  L^2(\Omega_i^n, (\mu_i)^{\tensor n})$ satisfying
  \begin{itemize}
  \item For every $i \in [k]$, $\|f_i\|_{\infty} \le 1$.
  \item For every $d \in [n]$, $\|f_i^{\ge d}\|_2^2 \le (1-\gamma)^d$.
  \item For every $\sigma \in \Z_q^n$, $|\hat{f_1}(\sigma)| \le \delta$.
  \end{itemize}
  Then
  $$
  \noiseprod{f_1, \ldots, f_k} \le \epsilon.
  $$
\end{question}

An affirmative answer to \questref{extension} would also have
consequences for completely dropping the degree requirement under
additional conditions on the marginal distributions.

In particular, for marginal distributions whose support is
\emph{connected} in the sense described in \sectref{smoothness}, by~\cite{Mossel:09} it is
known that applying a small amount of noise to each of the functions
$f_1, \ldots, f_k$ does not change $\noiseprod{f_1, \ldots, f_k}$ by
much.

Since applying noise gives exponentially decaying Fourier
tails, an affirmative answer to \questref{extension} implies that for
connected marginal distributions, the condition on the Fourier degree
of the functions can be dropped completely.

The statement of \questref{extension} allows for much weaker bounds on
the error $\epsilon$ than we had in \theoremref{corrbound_main},
where the error bound was of the form $\lambda(d, \delta) \cdot
\prod_{i=2}^k \|f_i\|_2$ (where $\lambda(d,\delta) = \delta C^d$).
One cannot hope for such a strong error bound in the setting of
\questref{extension} (with $\lambda(d, \delta)$ replaced by some
function $\lambda(\gamma, \delta)$ depending on the rate of decay of
the Fourier tails, rather than the degree), as illustrated by the
following example communicated to us by Hamed Hatami, Shachar Lovett,
Alex Samorodnitsky and Julia Wolf: consider a pairwise independent
distribution $\mu$ on $\{0,1\}^k$ in which the first $\approx \log k$
bits are chosen uniformly at random, and the remaining bits are sums
of different subsets of the first $\log k$ bits.  This distribution is
not connected in the sense described above, but that can easily be
arranged by adding a small amount of noise to $\mu$, which will not
have any significant impact on the calculations which follow.  Let $f:
\{0,1\}^n \rightarrow \{0,1\}$ be the function which returns $1$ on
the all-zeros string, and $0$ otherwise.  Then, one has that
$$
\noiseprod{f, \ldots, f} = \Pr[X_1 = \ldots = X_k = 0] \approx 2^{-n \log k},
$$ whereas $\|f\|_2 = 2^{-n/2}$ and hence the product
$\prod_{i=2}^k \|f\|_2$ equals $2^{-n(k-1)/2}$ so that
$$
\lambda(\gamma, \delta) \cdot \prod_{i=2}^k \|f\|_2 =
\lambda(\gamma, \delta) 2^{-n(k-1)/2} \ll \noiseprod{f, \ldots, f}.
$$
One may argue that it is more reasonable to bound $\noiseprod{f_1,
\ldots, f_k}$ in terms of e.g.\ the $\ell_k$ norms of the $f_i$'s rather
than the $\ell_2$ norms.  We do not know of any counterexample to such
a strengthening of \questref{extension}.

\subsection{A partial solution to ~\questref{extension} by Hamed Hatami}
We were recently informed by Hamed Hatami (personal communication) that \questref{extension}
admits a positive answer in the case where $\mu$ is the support of $(L_1,\ldots,L_k)$ where the $L_i$ are
distinct linear forms over the same additive groups.
This follows since given a value of $\eps$ we may choose $d$ large enough so that $\| f_1^{\geq d} \|_2 \leq \eps/2$ so applying Cauchy-Schwartz yields that
\[
\noiseprod{f_1^{\geq d}, f_2, \ldots, f_k}\leq \eps/2.
\]
On the other hand applying the Gowers-Cauchy-Schwartz inequality using the fact that $f_2,\ldots,f_k$ are bounded one can obtain:
\[
\noiseprod{f_1^{< d}, f_2, \ldots, f_k} \leq \|f_1^{<d}\|_{U^{k-1}},
\]
Thus choosing $\delta(\eps,d)$ sufficiently small and using the main result of the paper we obtain that the last quantity is at most $\eps/2$.

\section{Acknowledgments}

We are grateful to Hamed Hatami, Shachar Lovett, Alex Samorodnitsky
and Julia Wolf for communicating us the example in
\sectref{relaxeddegree}.  We would also like to thank Madhur Tulsiani
for stimulating and helpful discussions.  We thank Hamed Hatami for providing the partial solution
to \questref{extension}.
Finally, we are
grateful to Ryan O'Donnell and the two anonymous referees for their many comments which
improved the presentation of the paper.

\bibliographystyle{abbrv}
\bibliography{all,my,references}

\end{document}